\documentclass[11 pt]{amsart}
\usepackage{epsfig,graphics,amsmath,amssymb,amsthm}
\newtheorem{Theorem}{Theorem}[section]

\newtheorem{Corollary}[Theorem]{Corollary}
\newtheorem{Proposition}[Theorem]{Proposition}
\newtheorem{Lemma}[Theorem]{Lemma}

\theoremstyle{definition}
\newtheorem{Definition}[Theorem]{Definition}

\newtheorem{Remark}[Theorem]{Remark}

\begin{document}

\title[Outer automorphisms of mapping class groups]
{Outer automorphisms of mapping class groups of nonorientable
surfaces}

\author{Ferihe Atalan}
\address{Department of Mathematics, Atilim University,
06836 \newline Ankara, TURKEY} \email{fatalan@atilim.edu.tr}

\date{\today}
 \pagenumbering{arabic}

\begin{abstract}Let $N_{g}$ be the connected closed nonorientable surface of
genus $g\geq 5$ and ${\rm Mod}(N_{g})$ denote the mapping class group of
$N_{g}$. We prove that the outer automorphism group of ${\rm
Mod}(N_{g})$ is either trivial or $\mathbb{Z}$ if $g$ is odd, and
injects into the mapping class group of sphere with four holes
if $g$ is even.

\end{abstract}

\maketitle
\section{Introduction and statement of results}

Let $N_{g,n}$ be a connected nonorientable surface of genus $g$ with
$n$ holes (boundary components) and let ${\rm Mod}(N_{g,n})$ denote
the mapping class group of $N_{g,n}$, the group of isotopy classes
of all diffeomorphisms $N_{g,n} \to N_{g,n}$. Our work is motivated by
the following theorem of N. V. Ivanov.

\begin{Theorem}(Ivanov)\label{Ivanov}
Let $M = {\rm Mod}(S)$, ${\rm PMod}(S)$, ${\rm Mod}^{*}(S)$ or ${\rm
PMod}^{*}(S)$. If $S$ is neither a sphere with at most $4$ holes,
nor a torus with at most $2$ holes, nor a closed surface of genus
$2$, then every automorphism of $M$ is given by the restriction of
an inner automorphism of ${\rm Mod}^{*}(S)$. In particular, ${\rm
Out}\,({\rm Mod}(S))$ is a finite group and moreover, ${\rm Out}
\,({\rm Mod}^{*}(S)) = 1$, ${\rm Out}({\rm Mod}(S))= \mathbb{Z} / 2
\mathbb{Z}$.
\end{Theorem}

In the next section we study the orientation double cover of a
nonorientable surface and we investigate relation between mapping
class groups of the nonorientable surface and its orientation double
cover. Furthermore, we describe maximal abelian subgroups of the mapping
class group of nonorientable surfaces and construct some special type of
abelian subgroup of the mapping class group of nonorientable surface.

In Section~\ref{Dehn}  we characterize the Dehn twists about circles on a
closed nonorientable surface using the ideas of Ivanov's analogous
work on the orientable surface \cite{I1}.

In Section~\ref{OutAut}, we prove that the outer automorphism of
$ {\rm {Mod}}(N_{2g+1})$ injects into $\mathbb{Z}$ if $g \geq 2$, on the
other hand, the outer automorphism of ${\rm {Mod}}(N_{2g})$ injects
into the mapping class group of sphere with 4 holes if $g \geq 3$.

Finally, we compute the outer automorphism of mapping class groups
of some sporadic cases using the presentations of mapping class groups
of some nonorientable surfaces not covered in our theorems.

\section{Preliminaries and Notations}\label{Prelim}

Let $\Sigma_{g,n}$ and $N_{g,n}$ denote the connected orientable surface
of genus $g$ with $n$ holes and the connected nonorientable surface of
genus $g$ with $n$ holes, respectively. Recall that the genus of a
nonorientable surface is the maximum number of projective planes in
a connected sum decomposition.

\subsection{The orientation double cover of $N_{g,n}$} Let
$p: \Sigma_{g-1,2n} \rightarrow N_{g,n}$ be the orientation double cover of
the nonorientable surface $N_{g,n}$ with the Deck transformation
$\sigma:\Sigma_{g-1,2n}\rightarrow \Sigma_{g-1,2n}$. Then $p_{*}(\pi_{1}(\Sigma_{g-1,2n}, x))$
is a normal subgroup of index $2$ in $\pi_{1}(N_{g,n}, p(x))$ and the deck
transformation group is $\langle \sigma \mid \sigma^{2}=identity
\rangle \cong {\mathbb Z}_2$. If $F'$ is a diffeomorphism of $\Sigma_{g-1,2n}$ such
that $\sigma \circ F' = F' \circ \sigma$, then $F'$ induces a diffeomorphism $F$ on
$N_{g,n}$.

Let $Diff(\Sigma_{g-1,2n})^{\langle \sigma \rangle} = \{ F' \in Diff(\Sigma_{g-1,2n})
\mid F'=\sigma \circ F' \circ \sigma^{-1} \}$. So, we have a
homomorphism $Diff(\Sigma_{g-1,2n})^{\langle \sigma \rangle}
\rightarrow Diff(N_{g,n})$. Moreover, this homomorphism is onto. Indeed, we
have an exact sequence $$0 \rightarrow \mathbb{Z}_{2} \rightarrow
Diff(\Sigma_{g-1,2n})^{\langle \sigma \rangle} \rightarrow
Diff(N_{g,n}) \rightarrow 0.$$ On the other hand, any diffeomorphism $F$ of the
nonorientable surface $N_{g,n}$ has a unique lift to an orientation preserving diffeomorphism
$F':\Sigma_{g-1,2n}\rightarrow \Sigma_{g-1,2n}$. Hence, the above exact
sequence splits and therefore we can identify $Diff(N_{g,n})$ with the subgroup
$Diff_+(\Sigma_{g-1,2n})^{\langle\sigma \rangle}$, the subgroup of orientation
preserving diffeomorphisms of $\Sigma_{g-1,2n}$ commuting with the deck transformation $\sigma$.

Indeed, any isotopy of the nonorientable surface $N_{g,n}$ lifts to the surface
$\Sigma_{g-1,2n}$. Moreover, if the Euler characteristic of
the surface is negative (hence it carries a complete hyperbolic metric) then any isotopy of
orientation preserving diffeomorphisms of $\Sigma_{g-1,2n}$, whose end maps commute with the deck
transformation $\sigma$, can be deformed into one which commutes with $\sigma$, and
hence descends to an isotopy of $Diff(N_{g,n})$ (c.f. Lemma 5 in \cite{Y}).  Therefore, we can identify
the mapping class group ${\rm {Mod}}(N_{g,n})$ as the subgroup
${\rm {Mod}}(\Sigma_{g-1,2n})^{\langle\sigma \rangle}$, the subgroup of orientation preserving
mapping classes of $\Sigma_{g-1,2n}$ commuting with the mapping class of the deck transformation $\sigma$.

\subsection{Circles}
If $a$ is a circle on $N_{g,n}$, by which we mean a simple closed
curve, then according to whether a regular neighborhood of $a$ is an
annulus or a M\"obius strip, we call $a$ two-sided or one-sided
simple closed curve, respectively.

We say that a circle is nontrivial if it bounds neither a disc with
at most one puncture nor annulus together with a boundary component,
nor a M\"obius band on $N_{g,n}$.

If $a$ is a circle, then we denote the complement of $a$ by
$N_{g,n}^{a}$ the surface obtained by cutting $N_{g,n}$ along $a$. A
circle $a$ is called nonseparating if $N_{g,n}^{a}$ is connected and
otherwise, it is called separating.

Following Ivanov \cite{I1}, we will call a one dimensional submanifold
$\mathcal{C}$ of a surface $S$ a system of circles if $\mathcal{C}$ consists of
several pairwise nonisotopic circles. By $S^\mathcal{C}$ we will denote the
surface $S$ cut along all the circles in the system $\mathcal{C}$.

\begin{Definition}
A finite sequence of embedded two-sided circles $\{c_{1},\cdots,c_{k}\}$ on
the any surface is called a chain if the geometric
intersection number of the
isotopy classes $\gamma_{i}$ and $\gamma_{j}$, satisfies
$i(\gamma_{i},\gamma_{j})=\delta_{i+1,j}$, for all $i,j$,
where the geometric intersection number $i(\gamma_{i},\gamma_{j})$
defined to be the infimum of the cardinality of $d_{i} \cap d_{j}$ with
$d_{i} \in \gamma_{i}=[c_i]$, $d_{j} \in \gamma_{j}=[c_j]$.
\end{Definition}

\subsection{Periodic, reducible, and pseudo-Anosov elements}
For an orientable surface $\Sigma_{g,n}$ we have the classical definition below.
\begin{Definition}
Let $F:\Sigma_{g,n} \to \Sigma_{g,n}$ be any diffeomorphism.

\hspace*{0.4cm} $\bullet$ $F$ is called periodic if it is isotopic to some
diffeomorphism $\tilde{F}$  with $\tilde{F}^k = id$ for
some $k \neq 0$,

\hspace*{0.4cm} $\bullet$ $F$ is called reducible if there is a
system of circles $\mathcal{C} \neq \emptyset$ on $\Sigma_{g,n}$ and a
diffeomorphism $\tilde{F}$, isotopic to $F$, with
$\tilde{F}(\mathcal{C})=\mathcal{C}$,

\hspace*{0.4cm} $\bullet$ otherwise, $F$ is called pseudo-Anosov.

If $F$ is periodic, or reducible or pseudo-Anosov, its isotopy class
$f$ is called periodic element, reducible element, pseudo-Anosov
element, respectively.
\end{Definition}

It has been shown in \cite {Y} that one can define periodic, reducible
and pseudo-Anosov diffeomorphisms for nonorientable surfaces exactly
in the same way.  Moreover, if $F$ is a diffeomorphism of the
nonorientable surface $N_{g,n}$ with its unique orientation preserving
lift $F'$ to the orientation double cover $\Sigma_{g-1,2n}$, then $F$ is periodic,
reducible or pseudo-Anosov if and only if $F'$ is periodic, reducible or
pseudo-Anosov, respectively.

\subsection{Pure elements}
\begin{Definition}
A diffeomorphism $F : N_{g,n} \to N_{g,n} $ is called pure if for
some system of circles $\mathcal{C}$ the following condition is satisfied:
All points of $\mathcal{C}$ and $\partial N_{g,n}$
are fixed by  $F$, $F$ does not interchange components of $N_{g,n}^\mathcal{C}$
and it induces on every component of $N_{g,n}^\mathcal{C}$ a diffeomorphism,
which is either isotopic to a pseudo-Anosov one, or to the identity.
\end{Definition}

\subsection{Abelian subgroups of $\Gamma'(m)$}
Let $\Gamma(m)$, where $m \in \mathbb{Z}$, $m>1$, be the kernel of the
natural homomorphism $${\rm Mod}(\Sigma_{g-1,2n}) \rightarrow {\rm Aut}
\,(H_{1}(\Sigma_{g-1,2n}, \mathbb{Z} / m \mathbb{Z})).$$ Then
$\Gamma(m)$ is a subgroup of finite index in ${\rm
Mod}(\Sigma_{g-1,2n})$. Let $\Gamma'(m)=\Gamma(m) \cap {\rm
Mod}(N_{g,n})$, regarding ${\rm Mod}(N_{g,n})$ as a subgroup of ${\rm
Mod}(\Sigma_{g-1,2n})$ as described above.

It is well known that $\Gamma(m)$ and hence $\Gamma'(m)$ consist of pure
elements only provided that $m>2$, (\cite{I1}). Throughout the paper, we will
consider $\Gamma(m)$ for only $m>2$.

Following Ivanov, suppose that $\mathcal{C}$ is a system of circles on
$N_{g,n}$ and a diffeomorphism $\bar{F_{P}}: P \to P$ is fixed on each component
$P$ of $N_{g,n}^\mathcal{C}$. Moreover, suppose that $\bar{F_{P}}$ is
fixed at all points of the boundary of $P$ and is either isotopic to the
identity or to a pseudo-Anosov map on $P$. We extend $\bar{F_{P}}$ to $F_{P}: N_{g,n}
\to N_{g,n}$ by the identity. Let the subgroup $\Lambda$ of ${\rm
Mod}(N_{g,n})$ be generated by $f_{P}$'s, which are the isotopy classes of
such $F_{P}$'s, and all Dehn twists about two-sided circles in $\mathcal{C}$. The
subgroup $\Lambda$ is abelian. Indeed, in the orientable surface case, Ivanov showed
that any maximal abelian subgroup of $\Gamma(m)$ is of this type (\cite{I1}). The
analogous result holds in the nonorientable case also:

\begin{Lemma}\label{lambdasubgroup1}
Every abelian subgroup of  \ $\Gamma'(m)$ is contained in some
abelian subgroup constructed as in the above paragraph.
\end{Lemma}

\begin{proof}
Consider an abelian subgroup $K \leq \Gamma'(m)$. Regarding $K\leq \Gamma(m)$
let $\mathcal{C}$ be a reduction system for $K$ described by Ivanov for
orientable surfaces \cite{I1}. For any
$x \in C \subseteq \mathcal{C}$ and $f\in K$, we have
$f(\sigma(x))=\sigma(f(x))=\sigma(x)$, because each $f \in K$ is invariant
under the involution $\sigma: \Sigma_{g-1,2n} \to \Sigma_{g-1,2n}$.  So we have
$\sigma(\mathcal{C})=\mathcal{C}$.  Let $C\in \mathcal{C}$ be a circle in the system
with $\sigma(C)=C$. Since $\sigma$ has no fixed points,
the restriction of $\sigma$ to the circle $C$ is an orientation preserving
involution on that circle.  So each circle of the system $\mathcal{C}$
maps onto a circle in $N_{g,n}$ and hence the system $\mathcal{C}$ gives a system
$\mathcal{C}'$ on $N_{g,n}$. Now $K \subseteq \Lambda$, where $\Lambda$
is an abelian subgroup corresponding to the system of circles
$\mathcal{C}'$, as described in the above paragraph.
\end{proof}

The system of circles $\mathcal{C}'$ constructed in the above proof will be called
a reduction system for the abelian subgroup $K$.

Let $\Gamma$ be a subgroup of finite index of $\Gamma(m)$.
Then, the its $\sigma$ invariant part,
$\Gamma'=\{x \in \Gamma \mid \sigma x \sigma^{-1} = x \}$, is also
finite index subgroup of $\Gamma'(m)$.

\begin{Lemma}\label{lambdasubgroup2}
The center of every subgroup of finite index of \,$\Gamma'(m)$ is
trivial.
\end{Lemma}

\begin{proof}
The centralizer in $\Gamma(m)$ of a pseudo-Anosov element is
isomorphic to $\mathbb{Z}$, (\cite{I1}). Therefore, the centralizer
in $\Gamma'(m)$ of a pseudo-Anosov element is also isomorphic to
$\mathbb{Z}$. From this, we obtain that the center of every subgroup
of finite index of  $\Gamma'(m)$ is trivial.
\end{proof}

Stukow's result ~\cite{SM} on Dehn twists on nonorientable surfaces
shows that Dehn twists on nonorientable surfaces enjoy the same
properties with those on orientable surfaces. For instance, two Dehn twists commute
if and only if their circles can be chosen disjoint; they satisfy the braid relation if
and only if their circles meet transversally at one point.

Before we go further we give the following remark.

\begin{Remark}
{\bf a)} The pure mapping class group of real projective plane with less than
or equal to two holes is finite.\\
{\bf b)} The pure mapping class group of the Klein bottle is finite and that of the
Klein bottle with one hole has no pseudo-Anosov element (see Proposition~\ref{sporadic cases}).\\
{\bf c)} The pure mapping class group of Klein bottle with two holes contains
$\mathbb{Z}\bigoplus\mathbb{Z}$.\\
{\bf d)} The pure mapping class group of $N_{3,n}$ \ ($n\geq 0$)  contains
a copy of $\mathbb{Z}\bigoplus\mathbb{Z}$, though that of $N_3$ does
not support any pseudo-Anosov diffeomorphism (\cite{P}).
\end{Remark}

Using the above remark we deduce that the smallest surfaces which
support a pseudo-Anosov diffeomorphism in their pure
mapping class group which does not have an abelian subgroup of rank
two are $\Sigma_{0,4}$, $\Sigma_1$, $\Sigma_{1,1}$ or
$N_{1,3}$. Examples of pseudo-Anosov mappings on the surfaces
$\Sigma_{0,4}$, $\Sigma_1$ and $\Sigma_{1,1}$ are well known.
For the nonorientable surface $N_{1,3}$, the map $t_at_bt_c$, described
in the Figure~\ref{rp2-3bc}, is indeed pseudo-Anosov (Theorem 4.1 in
\cite{P}).

\begin{figure}[hbt]
  \begin{center}
    \includegraphics[width=5cm]{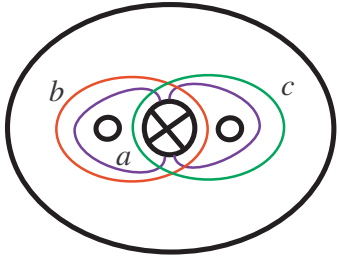}
    \caption{The real projective plane with three boundary components}
    \label{rp2-3bc}
  \end{center}
 \end{figure}

We will finish this section with the following result, whose proof will
be omitted.

\begin{Lemma}Let $K < \Gamma'(m)$ be a maximal abelian subgroup of $N_g$
of highest rank, where $g$ is at least three. Then the
$\mbox{rank} \ (K)=\frac{3g-r}{2}-3$ and each component of $N_g^{\mathcal{C}}$
is either sphere with three holes, sphere with four holes, torus with one hole,
or real projective plane with two or three holes, where $\mathcal{C}$ is a
minimal reduction system for $K$ and $r$ is $1$ or $0$ depending on whether
$g$ is odd or even.

Moreover, in both even and odd genus cases, one can choose maximal abelian
subgroups of highest rank so that they are generated by Dehn twists about
nonseparating circles.
\end{Lemma}

\bigskip
\section{Characterization of Dehn twists}\label{Dehn}

\subsection{Algebraic characterization of Dehn twists}

Throughout this section $\Gamma'$ is a finite index subgroup of $\Gamma'(m)$, $m>2$.

\begin{Theorem}\label{Chr-1}
Let $N_g$ be a compact closed connected nonorientable surface
of odd (respectively, even) genus $g \geq 5$ (respectively, $g \geq 6$). An element $f \in {\rm
Mod}(N_g)$ is a Dehn twist about a nonseparating circle (respectively, nonseparating
circle whose complement is nonorientable) if and only if the following conditions are satisfied:\\
\begin{itemize}

\item[i)] $C(C_{\Gamma'}(f^{n}))$ is isomorphic to
$\mathbb{Z}$, for any integer $n \neq 0$ such that $f^{n} \in
\Gamma'$.

\item[ii)] There is an abelian subgroup $K$ of rank
$\displaystyle\frac{3g-7}{2}$ (respectively,
$\displaystyle\frac{3g-8}{2}$) generated by $f$ and its conjugates freely.

\item[iii)] $f$ is a primitive element of $C_{{\rm Mod}(N_g)}(K)$.
\end{itemize}
\end{Theorem}

\begin{proof}

We will prove the odd genus case mainly following the proof of Ivanov in \cite{I1}:
The even genus case is basically the same with minor changes, which we will
emphasize in the when needed.

If the above conditions are satisfied, then we need to show that $f$
is a Dehn twist about a nonseparating circle.

Let us denote the number $\displaystyle\frac{3g-7}{2}$
($\displaystyle\frac{3g-8}{2}$, in the even genus case) by $s$.

If $f^{n}$ is equal to the identity, then $C_{\Gamma'}(f^{n}) =
\Gamma'$ and hence by the first condition, $C(\Gamma')$ is
isomorphic to $\mathbb{Z}$. However, this is a contradiction by
Lemma~\ref{lambdasubgroup2}. Therefore, $f^{n}$ is not equal to the
identity.

We will denote a minimal reduction system of $f^{n}$ by
$\mathcal{C'}$, and the subgroup generated by twists about two-sided
circles in $\mathcal{C'}$ by $G$. Set $G' = G \cap \Gamma'$. $G$ and
$G'$ are free abelian groups. First we will show that $G' \subset
C(C_{\Gamma'}(f^{n}))$. Let $g \in C_{\Gamma'}(f^{n})$. This implies
that $f^{n}$ and $g$ commute and hence $f^{n}$ and $g$ belong to
some abelian subgroup $\Lambda$ constructed as in the previous section.
Let $\mathcal{C}$ denote the system of
circles used in the construction of $\Lambda$. Therefore,
$\mathcal{C'}$ is isotopic to a subsystem of $\mathcal{C}$, because
$f^{n} \in \Lambda$. Thus, we have $G \subset \Lambda$. So, each
element of $G$ commutes with all $g \in \Lambda$. Therefore, we obtain
that $G \subseteq C_{{\rm Mod} (N_g)}(C_{\Gamma'}(f^{n}))$ and we
obtain that $G' \subset C_{\Gamma'}(C_{\Gamma'}(f^{n})) \subset
C(C_{\Gamma'}(f^{n}))$.

By the assumption that $C(C_{\Gamma'}(f^{n}))=\mathbb{Z}$, we have
$\mathcal{C'}$ has no or just one
two-sided circle. Suppose that $\mathcal{C'}$ has no two-sided
circle. Say $\mathcal{C'} = \{c_{1}, \cdots, c_{k} \}$, where
$c_{i}$ is a one-sided circle for all $i$. Then
$N_g^{\mathcal{C'}}$ is connected and the restriction $f_{\mid
N_g^{\mathcal{C'}}}$ is either the identity or a pseudo-Anosov.

If $f_{\mid N_g^{\mathcal{C'}}}$ is the identity, then $f$ must
be a product of Dehn twists about some circles in $\mathcal{C'}$
which is not possible since each $c_{i}$ is one-sided. Therefore, $f_{\mid
N_g^{\mathcal{C'}}}$ is a pseudo-Anosov. However, the maximal
abelian group containing $f$ has rank one. By the assumption $g\geq 5$ and thus
the rank is $s\geq 4 > 1$. This is a contradiction.  (Similarly, in the even genus
case, we have $s\geq 5 > 1$.) So, $\mathcal{C'}$ has exactly one two-sided circle
and so $\mathcal{C'}= \{c_{1}, \cdots, c_{k}, a\}$,
where $a$ is a two-sided circle and each $c_i$ is one-sided.

Let  $D$ be the group generated by $f^{n}$ and the twist about $a$ and
set $D' = D \cap \Gamma'$. Then, we have $D' \subset
C(C_{\Gamma'}(f^{n}))$ and hence $D'$ is isomorphic
to $\mathbb{Z}$. It follows that there are no pseudo-Anosov
diffeomorphisms in the Thurston normal form of $f^{n}$ and thus $f^{n}$
is a power of the Dehn twist $t_{a}$ about $a$. So, $f^{n} = t_{a}^{m}$
for some integer $m$.

Now we need to show that $a$ is nonseparating. We assume
that $a$ is a separating circle. So, $N_g^{a} = N_{1} \cup
N_{2}$ such that $\chi_{j}=\chi(N_{j})$ and
$\chi_{1}+\chi_{2}=\chi(N_g)$. Since $a$ is nontrivial we have
$\chi_{j} < 0$ for $j=1,2$. Without loss of generality, assume
$\chi_{1} \geq \chi_{2}$.

Let $f_{1}, \cdots, f_s$ be the elements conjugate to $f$
generating the abelian subgroup $K$ in the statement of the theorem.
Then $f^{n}_{i} = t_{a_{i}}^{m}$, where $a_{i}$ is a circle.
Similarly, $a_{i}$ separates the surface into two components,
say $N_{1}^{i}$ and $N_{2}^{i}$, which are diffeomorphic to $N_1$ and
$N_2$, respectively. By the structure of abelian groups, the circles $a_{i}$
can be chosen to be disjoint and pairwise nonisotopic. Moreover, $N_{1}^{i}$
is diffeomorphic to $N_{1}^{j}$ and $a_{i}$ is not isotopic to $a_{j}$
for any $i \neq j$. So, $N_1^{i}$ cannot be contained inside $N_{1}^{j}$.
Therefore, if $N_{1}^{i} \cap N_{1}^{j} \neq \emptyset$, then $N_{1}^{i} \cup
N_{1}^{j} = N_g$.

Since $a_{i}$ is not isotopic to $a_{j}$, we
see that $\chi(N_{1}^{i} \cap N_{1}^{j}) < 0$. Hence, we see that
$\chi(N_g) = \chi(N_{1}^{i}) + \chi(N_{1}^{j}) - \chi(N_{1}^{i}
\cap N_{1}^{j}) = \chi_{1} + \chi_{1} - \chi(N_{1}^{i} \cap
N_{1}^{j}) > \chi_{1} + \chi_{1} \geq \chi_{1} + \chi_{2} =
\chi(N_g)$ which is a contradiction. So, $N_{1}^{i} \cap
N_{1}^{j} = \emptyset$, in other words, the surfaces $N_{1}^{i}$ are
all disjoint. Let $N^{0} = Cl(N_g \setminus (N_{1}^{1} \cup \cdots \cup N_{1}^s))$.
The fact that $a_{i}$ are being pairwise nonisotopic implies that $N^{0}$ cannot
an annulus.  Also, $N^{0}$ is not a sphere, a torus, a Klein bottle,
a projective plane, a disk or a M\"obius band. Thus, $\chi(N^{0}) < 0$.

Now, $2-g = \chi(N_g)= \chi(N^{0}) + \chi(N_{1}^{1}) + \cdots + \chi(N_{1}^s)$.
 So, $2-g \leq s \chi(N_{1}^{1})$. Then, $g-2 \geq s
\mid \chi(N_{1}^{1}) \mid$. So we get $2g-4 \geq 3g-7$, which implies $3 \geq g$.
This is a contradiction. Thus, $a$ is a nonseparating circle. (For even genus, we would have
$2g-4 \geq 3g-8$, which would yield $4 \geq g$, a contradiction to the assumption that
$g\geq 6$.)

We take an element $\psi \in C_{{\rm Mod}(N_g)}(K)$. $\Psi :
N_g \to N_g$ denotes a diffeomorphism in the isotopy class
$\psi$. The group $K$ is generated by $f_{1}, \cdots, f_s$ and
$f_{i} = t_{a_{i}}^{m_{i}}$ as before. Since $\psi \in C_{{\rm
Mod}(N_g)}(K)$, $\psi f_{i} \psi^{-1} = f_{i}$ and thus $\psi
f_{i}^{n} \psi^{-1} = f_{i}^{n}$. In other words, $\psi
t_{a_{i}}^{m_{i}} \psi^{-1} = t_{a_{i}}^{m_{i}}$ and so
$t_{\Psi(a_{i})}^{m_{i}} = t_{a_{i}}^{m_{i}}$. So, replacing $\Psi$
by an isotopy we can assume that $\Psi(a_{i})=a_{i}$. Let
$\mathcal{\bar{C}} = \{a_{1} \} \cup \cdots \cup \{a_s\}$.

Since $s$ is the maximal number of pairwise nonisotopic disjoint
two-sided circles on $N_g$, every component of $N_g^{\mathcal{\bar{C}}}$ is a
disc with two holes or a one holed M\"obius band. Total number of the
boundary components of the components of $N_g^{\mathcal{\bar{C}}}$
is $2s$ and hence all the components of
$N_g^{\mathcal{\bar{C}}}$ except one must be pair of pants.
(In the even genus case, $s+1$ is the maximal number of nonisotopic
two-sided circles on $N_g$. Now, similar arguments prove that all
the components $N_g^{\mathcal{\bar{C}}}$ but one are
discs with two holes, and the remaining component is a sphere with four
holes.)

Suppose that $\Psi$ sends one of these components to another one.
Since $\Psi$ fixes the circles(boundary components) of
$N_g^\mathcal{\bar{C}}$, we see that $N_g= \bar P \cup \bar R$ where
$\Psi(P)=R$. This is a contradiction to the assumption that
$N_g$ is a nonorientable surface of odd genus. So $\Psi$ does
not permute the components. Hence, $\Psi$ is a composition of twists
about $a_{i}$. (Again, in the even genus case, since there is only one
sphere with four holes component of $N_g^{\mathcal{\bar{C}}}$, we see that
$\Psi$ cannot interchange the components.)

Now note that, $f \in C_{{\rm Mod}(N_g)}(K)$ because $f \in K$,
and $K$ is abelian. Without loss of generality, suppose that $f=f_{1}$. So,
$f^{n} = f_{1}^{n} = t_{a_{1}}^{m}$. So by the above paragraph, we
obtain that $f$ is a power of $t_{a_{1}}$; $f=t_{a_{j}}^{l}$ for
some $l$, because $K$ is free abelian. Therefore, $f_{i}=t_{a}^{l}$
and $K$ is generated by $t_{a_{i}}^{l}$. Thus, $t_{a}$ is in
$C_{{\rm Mod}(N_g)}(K)$. By the last condition, $f$ is primitive
and hence, $f = t_{a}$.

In the even genus case, one has to prove that the complement of $a$
is nonorientable.  Assume that the complement of $a$ is orientable.
Since each $t_{a_i}$ is conjugate to $t_a$ the complement of each $a_i$ is
also orientable. It is easy to see that each $a_i$ is a ${\mathbb Z}_2$
homology cycle Poinc\'are dual to the first Stiefel-Whitney class
$\omega_1$ of the surface $N_g$. In particular, for any $1\leq i,j \leq s$,
$i\neq j$, $a_i$ and $a_j$ form a boundary of a subsurface $S_{i,j}$ of $N_g$.
Moreover, since the complement of each $a_i$ is orientable, $S_{i,j}$ is an
orientable surface of genus at least one with two boundary components.
By relabeling $a_i$'s we may assume that the interior of each $S_{i,i+1}$
(where we set $S_{s,s+1}=S_{s,1}$) does not intersect any $a_k$. Hence,
$S_{i,i+1}\cap S_{j,j+1}=\emptyset$, if and only if $|i-j|>1$.  Now,
$N_g=\cup_{1\geq i \geq s}S_{i,i+1}$, and therefore we get $g\geq 2s+2=3g-8+2=3g-6>g$,
a contradiction since $g\geq 6$. Hence, the complement of $a$ must be nonorientable.

The other direction of the theorem is straight forward and left to the reader
(see also \cite{I1}).
\end{proof}

Similar arguments prove the following theorems (compare with the
corresponding result in \cite{I1}):

\begin{Theorem}\label{Chr-2.1} Let $g\geq 5$ be an odd integer.
An element  $f \in {\rm Mod}(N_g)$ is a Dehn twist
if and only if the following conditions are satisfied:

\begin{itemize}

\item[i)] $C(C_{\Gamma'}(f^{n}))$ is isomorphic to
$\mathbb{Z}$, for any integer $n \neq 0$ such that $f^{n} \in
\Gamma'$.

\item[ii)] There exists a free abelian subgroup $K$
of ${\rm Mod}(N_g)$ generated by $f$ and $\displaystyle\frac{3g-9}{2}$ twists about
two-sided nonseparating circles such that $rank \ (K)=\displaystyle\frac{3g-7}{2}$.

\item[iii)] $f$ is a primitive element of $C_{{\rm Mod}(N_g)}(K)$.

\end{itemize}
\end{Theorem}

For Dehn twists on even genus surfaces we have,

\begin{Theorem}\label{Chr-2.2} Let $g\geq 6$ be an even integer.
An element  $f \in {\rm Mod}(N_g)$ is a Dehn twist if and
only if the following conditions  are satisfied:
\begin{itemize}
\item[i)] $C(C_{\Gamma'}(f^{n}))$ is isomorphic to
$\mathbb{Z}$, for any integer $n \neq 0$ such that $f^{n} \in
\Gamma'$.

\item[ii)] There exists a free abelian subgroup $K$
of ${\rm Mod}(N_g)$ generated by $f$ and $\displaystyle \frac{3g-10}{2}$
Dehn twists about nonseparating circles whose complements are nonorientable,
such that $rank \ (K)=\displaystyle \frac{3g-8}{2}$.
\item[iii)] $f$ is a primitive element of $C_{{\rm Mod}(N_g)}(K)$.

\end{itemize}
\end{Theorem}

\begin{Remark}\label{Ch-3.2}
{\bf 1)} Let $\Phi:{\rm Mod}(N)\rightarrow {\rm Mod}(N)$ be an automorphism, where
the genus of $N$ is at least five. Since $\Gamma'$ is a subgroup of $\Gamma'(m)$ of
finite index the subgroup $\Phi(\Gamma')\cap \Gamma'(m)$ is still a finite index
subgroup of $\Gamma'(m)$. Hence the above theorems give an algebraic characterization
of a Dehn twist. Therefore, we can also characterize algebraically
the subgroup $T$ of  ${\rm Mod}(N)$ generated by all Dehn twists.\\
{\bf2)} A two-sided embedded circle on a closed nonorientable surface of odd genus
is either separating or nonseparating with nonorientable complement.
Therefore, by Theorem~\ref{Chr-1} and Theorem~\ref{Chr-2.1} an automorphism of the
mapping class group respects the topological type of the circles of the Dehn twists.

However, in case of nonorientable surfaces of even genus, Theorem~\ref{Chr-1}
and Theorem~\ref{Chr-2.2} are not strong enough to distinguish a Dehn twist
about a separating circle from a Dehn twist about a nonseparating circle with orientable
complement. To make sure that automorphisms respect the topological type of the circles
the Dehn twists are about, we need to study chains of Dehn twists.
\end{Remark}

\subsection{Chains of Dehn twists}
Now we would like to discuss an application of the above theorems to the
chains of Dehn twists, which we will use in the study outer
automorphisms of mapping class groups of nonorientable surfaces. Whenever we consider
a chain or a tree of Dehn twists we will choose the circles the Dehn twists are about so that
they intersect minimally (e.g. if the surface has negative Euler characteristic
then we put an hyperbolic metric on it and choose the unique geodesics from each isotopy
class).

Stukow's results (\cite{SM}) mentioned in Section $2$ together with the above theorems
imply that any chain of Dehn twists is mapped to a chain of Dehn twists under
an automorphism of the mapping class group.

On an orientable surface, a chain of circles is maximal if and only if it is
separating.  Moreover, in this case the chain has odd number of circles.
This is no longer true on nonorientable surfaces of even genera.  For example,
a two sided circle with orientable complement forms a maximal chain of length one;
however it is not separating. Nevertheless, we have the following result.

\begin{Lemma}\label{chain} Let $N_g$ be the nonorientable
surface of genus $g$.
\begin{itemize}
\item[\bf i)] If a chain of two sided nontrivial circles $\{c_1,\cdots c_l\}$ on the
closed surface $N_g$ is maximal then $l$ is an odd integer. Moreover, if $g$ is an
odd integer, then, the chain is maximal if and only if the chain separates the
surface into two components.

In case of surfaces of even genus $g$, the chain $\{c_1,\cdots c_l\}$ is maximal
if and only if it is ether separating, or nonseparating with orientable complement.

\item[\bf ii)]Let $\{c_1,\cdots c_{2k+1}\}$ be a separating maximal chain of two-sided
nontrivial circles on $N_g$.  Then it separates a disc if and only if there is no
Dehn twist which commutes with all of the Dehn twists about the circles of the chain,
but $t_{c_1}$.
\end{itemize}
\end{Lemma}

\begin{proof}
The proof of the part (i) is left to the reader.

For the statement in part (ii), Figure~\ref{2o}, which establishes the proof in two
special cases, indeed provides the proof of the ''if" part of the statement.

\begin{figure}[hbt]
  \begin{center}
    \includegraphics[width=10cm]{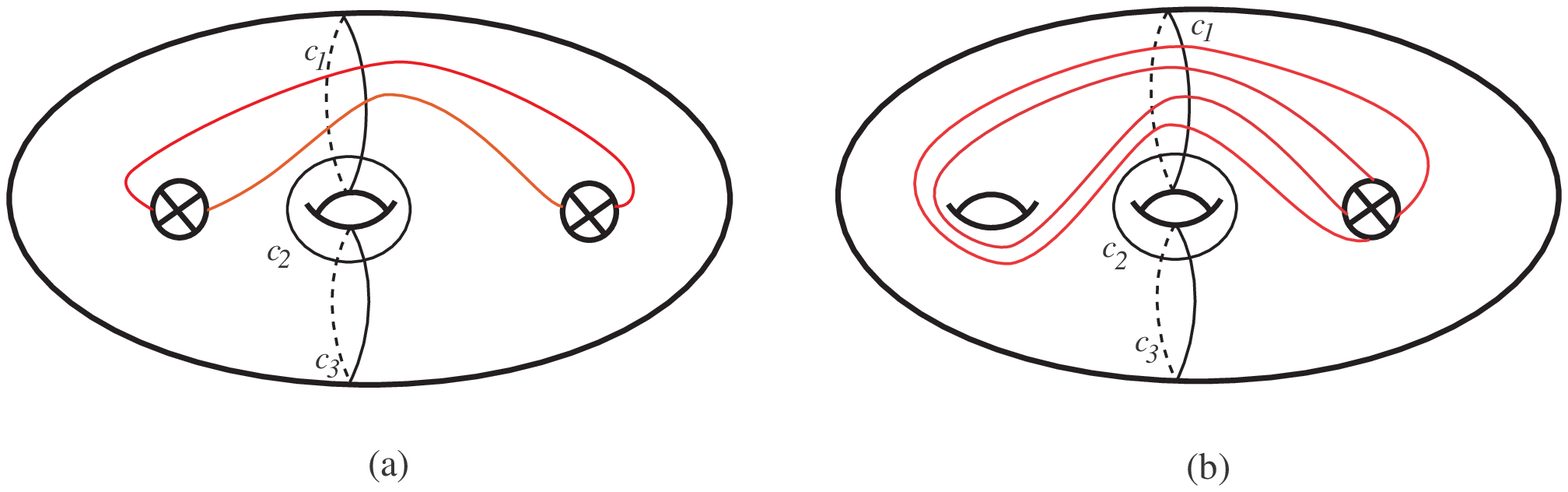}
    \caption{}
    \label{2o}
  \end{center}
 \end{figure}

For the other direction assume that there is a smooth curve $C$
which intersects the chain transversally at finitely points of only $c_1$ and
does not meet any other circle of the chain.  Let $S$ be any one of the
two subsurfaces obtained by cutting the surface $N_g$ along the chain.
Since $c_1$ is the first element of the chain, the $c_1$ part of the
boundary of $S$ is connected. Note that $S\cap C$ is a disjoint union of arcs.
Now, if $S$ is a disc we can isotope these arcs outside the surface $S$ so that
$C$ does not intersect the chain anymore.
\end{proof}

As mentioned in Remark~\ref{Ch-3.2} (1) we can use chains of Dehn twists to distinguish
algebraically Dehn twists about separating circles from Dehn twists about nonseparating
circles with orientable complements on the nonorientable surface of even genus $N_g$:
Indeed, the longest chain on $N_g$, which lies in the complement of a nonseparating circle
with orientable complement is $g-1$ (see also Lemma~\ref{Lemma-Y-hom}). On the other hand,
in the complement of a nontrivial separating circle the maximal length of a chain is again
$g-1$ but in this case the circle separates the surfaces into two components, a Klein bottle
with one boundary component and an orientable surface of genus $\displaystyle \frac{g-2}{2}$
with one boundary component. Hence, in the case of nonseparating circle with orientable
complement one can find one maximal chain of length $g-1$, whereas in the case of separating
circles one can find two maximal chains, which lie in different components (and thus
any element of one chain commutes with any element of the other chain) of lengths
$g-1$ and $1$.  These arguments prove the following:

\begin{Corollary}\label{Ch3.3} Let $g\geq 6$ be an even integer.
An automorphism of the mapping class group of $N_g$ maps a Dehn twist about a separating circle
to a Dehn twist about a separating circle.  Hence, an automorphism respects the topological types of
the circles of the Dehn twists.
\end{Corollary}

We will end this section with a result which will be useful in the latter part of the paper.

\begin{Corollary}\label{chains-bd} Let $g\geq 5$ be an integer.
An automorphism of the mapping class group of $N_g$ maps any maximal chain of Dehn twists on
$N_g$, which separates a disc, to a maximal chain of Dehn twists, which separates a disc.
\end{Corollary}

\begin{proof}
By Lemma~\ref{chain} it is enough to prove that the image of a separating chain must be
again separating. In the case of surfaces of odd genus a chain is separating if and only if
it is maximal, which is a property preserved by algebraic automorphisms. So we may assume that
$\{c_1,\cdots c_{2k+1}\}$ is a separating chain on the nonorientable surface $N_g$ of even genus
and the chain separates a disc from the surface.  Hence, $N_g$ cut along the chain is the
disjoint union of a disc and a nonorientable surface with one boundary component of genus
$g-2k$, which is a positive even integer because $g$ is even and the complement must
be nonorientable.  Therefore, there is an element, in the centralizer of the chain (i.e., it is
a $Y$ homeomorphism which commutes with all elements of the chain) which is not a product of
Dehn twists.

Now assume that the image $\{d_1,\cdots d_{2k+1}\}$, of this chain under an automorphism of the
mapping class group, is nonseparating. So by the above lemma, the complement of the chain is
an orientable surface, say $S$, with two boundary components of genus $\displaystyle \frac{g-2k-2}{2}$.
Now any mapping class, which is in the centralizer of the chain $\{d_1,\cdots d_{2k+1}\}$, has
a representative $F$, a diffeomorphism of $N_g$, which maps a tubular neighborhood $\nu$, of the chain
onto itself.  So, $F$ induces a diffeomorphism of both $\nu$ and $\overline{N_g- \nu}$. Since both
pieces are orientable $F$ is isotopic to a product of Dehn twists. However, this
contradicts to the fact that Dehn twists and hence their products are respected by automorphisms
of the mapping class group.
\end{proof}

Indeed, similar arguments can used to prove that the topological type of other separating
Dehn twists are preserved under automorphisms of the mapping class group.

\subsection{Characterization of Y-Homeomorphisms}
As an application of the previous results we will an algebraic characterization of
Y-homeomorphisms.  By definition a Y-homeomorphism on a nonorientable surface $N_{g,n}$ is a
diffeomorphism whose isotopy class is not a product of Dehn twists but its
square (composition with itself) is a Dehn twist about a separating circle which bounds a
one holed Klein bottle (c.f. see \cite{K,SB,SM}).

First we give some easy observations:
\begin{Lemma}\label{Lemma-Y-hom}
The longest maximal chain in $\rm Mod (N_{2g+1})$  (respectively, in $\rm Mod (N_{2g})$) has length
$2g+1$ (respectively, $2g-1$).
The longest chain in the centralizer of a $Y$-homeomorphism in $\rm Mod (N_{2g+1})$ (respectively,
in $\rm Mod (N_{2g})$) has length $2g-1$ (respectively, $2g-3$).
\end{Lemma}

\begin{proof}
We will give the proof of the odd genus case only.
The proof of the first statement is easy. For the second statement
suppose that there is a maximal chain in the centralizer of a
$Y$-homeomorphism in $\rm Mod (N_{2g+1})$ of length of $k \geq 2g$.
Let $\nu$ be a tubular neighborhood of the chain of circles the Dehn twists of the chain are about.
$\chi(\nu)=\chi(\vee_{k} S^1)=1-k$. A maximal chain should be
of odd length  and thus $k$ is odd. Then $\nu$ has two boundary components. Since $\nu$ is the
orientable surface of genus $(k-1)/2$ with two boundary components, $\Sigma_{(k-1)/2,2}$, we have
$\chi(N_{2g+1} \setminus \nu) = 2-(2g+1)-(1-k) = k-2g \geq 1 $. This
is a contradiction since $N_{2g+1} \setminus \nu$ must contain a one holed Klein bottle.
Hence, we get $k \geq 2g+1$. Finally, it is easy to see that chains of claimed
length exist in each case.
\end{proof}

Let $T\lhd \rm Mod (N_{g,n})$ denote the normal subgroup generated by Dehn twists.

\begin{Theorem}
An element $f \in \rm Mod (N_{2g})$ is a $Y$-homeomorphism if and
only if \\
i) $f \notin T$, $f^{2} = t_{e}$, where $t_{e}$ is a Dehn twist
about separating circle $e$.\\
ii) There is a chain of Dehn twist about nonseparating circles $D_1, \cdots, D_{2g-1}$
each of which has nonorientable complement, and
such that $D_{i} \in C_{\rm Mod
(N_{2g})}(t_{e})$,  for all $i \neq 2g-2$,
$D_{i} \in C_{\rm Mod (N_{2g})}(f)$, $i \leq 2g-3$.
\end{Theorem}

\begin{proof}
The tubular neighborhood of the union of the circles corresponding to the chain
$D_1,\cdots, D_{2g-3}$ is the orientable surface of genus $g-2$ with two boundary
components and hence, by assumption, $f$ is supported in $N_{2g} \setminus \Sigma_{g-2,2}$,
which is a copy of $N_{2,2}$ with boundary components, say $c_1$ and $c_2$.

Since $e$ is separating in $N_{2g}$ and is nontrivial, $e$ must separate
a Klein bottle. Also $D_{2g-1}$ and $e$ do not meet but they are
both contained in $N_{2,2}$ and hence we see that $D_{2g-1}$ is the unique
two-sided circle in the Klein bottle bounded by $e$. By assumption the
complement of $D_{2g-1}$ is nonorientable and hence complement of
Klein bottle must be nonorientable. Now, $f$ must be a product of
$t_{c_{1}}$, $t_{c_{2}}$ and some diffeomorphism supported in the
Klein bottle bounded by $e$. However, $f^{2}=t_{e}$ implies that $f$
has no $t_{c_{1}}$ and $t_{c_{2}}$ factors and hence must be
supported in the Klein bottle bounded by $e$.  Hence, it is a $Y$-homeomorphism.

The other part is easy and left to the reader.
\end{proof}

\begin{Theorem}
An element $f \in \rm{Mod}(N_{2g+1}) $ is a $Y$-homeomorphism if and
only if\\
i) $f \notin T$, $f^{2} = t_{e}$, where $t_{e}$ is a Dehn twist
about separating circle $e$.\\
ii) $C_{\rm Mod (N_{2g+1})}(t_{e})$ contains two chains of Dehn twists
$E_{1}$ and $D_{1}, \cdots, D_{2g-1}$ such that all but $E_1$ are about
nonseparating circles, satisfying
for all $i$, $[t_{D_{i}},t_{E_{1}}]=1$ and $D_{i} \in C_{\rm Mod
(N_{2g+1})}(f)$.\\
iii) The chain $D_{1}, \cdots, D_{2g-1}$ is separating but does
not separate a disc and the chain $E_1$ cannot be chosen of bigger length.
\end{Theorem}

\begin{proof}
Let $\nu$ be the tubular neighborhood of $\bigcup D_{i}$.
Then $\chi (\nu)=2-2g$ and hence,
$$\chi(N_{2g+1} \setminus \nu)=2-(2g+1)-(2-2g)=-1.$$ Also,
$\nu$ and hence $N \setminus \nu$ has two boundary components
(the chain $D_{1}, \cdots, D_{2g-1}$ has odd number of elements). Since
it is also a separating chain we see that
$N_{2g+1} \setminus \nu = N_{1,1}
\cup N_{2,1}$ or $N_{1,1}
\cup \Sigma_{1,1}$. So, $e$ must be the boundary of $N_{2,1}$ or
$\Sigma_{1,1}$. The second case is not realizable by the condition (iii)
because otherwise $e$ would bound a one holed torus and we could replace
the chain $E_1$ with a longer chain $E_1, \ E_2$ in the one holed torus.

Since $[f,D_{1}]=1$, we can isotope $f$ to fix $D_{1}$ pointwisely. Since
$[f,D_{2}]=1$ and $f(D_{1} \cap D_{2}) = D_{1} \cap D_{2}= {pt}$, we
can isotope $f$ to fix $D_{1} \cup D_{2}$ pointwisely. Similarly, we
can assume that $f$ fixes $D_{1} \cup \cdots \cup D_{2g-1}$
pointwisely. Hence, $f$ is supported outside of $\nu$, because $f$
cannot interchange the two sides of the tubular neighborhood. Moreover, since
$N_{2g+1} \setminus \nu = N_{1,1} \cup N_{2,1}$ and $N_{1,1}$ does not
support any nontrivial diffeomorphism which is identity on the
boundary we see that $f$ is supported in $N_{2,1}$. So,
$f$ is an element of $\rm{Mod}(N_{2,1})$ which is not a Dehn twist
so that $f^{2} = t_{e}$, where $e=\partial N_{2,1}$.
In particular, $E_1$ is the Dehn twist about the unique (up to isotopy)
two-sided circle, say $a$, in the one holed Klein bottle bounded by $e$.
Hence $f$ is a Y-homeomorphism.

The other direction is left to the reader.
\end{proof}

\begin{Remark}\label{Y-hom}
Using the notation of the above proof, let us continue our investigation of
Y-homeomorphisms in the mapping class group of the Klein bottle with one
boundary component. We have the presentation (\cite{K,SM}) $$\rm{Mod}
(N_{2,1})= \langle y, t_{a} \mid yt_{a}y^{-1}=t_{a}^{-1} \rangle =
\mathbb{Z} \rtimes \mathbb{Z},$$ where $t_a$ is the Dehn twist $E_1$.
One can easily see that $$C(\rm{Mod}(N_{2,1}))=\langle t_{b}
\rangle \cong \mathbb{Z},$$ where $t_{b}$ is $y^{2}$.

It is not hard to see that, any word in $\mathbb{Z} \rtimes \mathbb{Z}$
has the form $y^{m}t_{a}^{n}$ for $m, n \in \mathbb{Z}$, and this form is unique.
Let $x=y^{m}t_{a}^{n}$. Then $x^{2}=y^{m}t_{a}^{n}y^{m}t_{a}^{n}$. So,
$x^{2}$ equals $y^{2m} t_{a}^{2n}$ if $m$ is even, or $y^{2m}$ if
$m$ is odd. Hence, any  $x=yt_{a}^{n}$ satisfies the equation
$x^{2}=y^{2}$, $n \in \mathbb{Z}$. Since $t_{a}^{-n}yt_{a}^{n} =
yt_{a}^{2n}$ up to conjugation $y^{\pm}$ and $y^{\pm}t_{a}$ are the
only $Y$-homeomorphisms supported inside a punctured Klein bottle.
\end{Remark}

\section{Outer Automorphisms of ${\rm {Mod}}(N_g)$}\label{OutAut}

\begin{Theorem}\label{MainThm1}
$\rm Out( {\rm {Mod}}(N_{2g+1}))$ injects into $\mathbb{Z}$ if $g \geq 2$.
\end{Theorem}

\begin{proof}
Consider the maximal tree of Dehn twists in Figure~\ref{grp2}.

\begin{figure}[hbt]
  \begin{center}
    \includegraphics[width=8cm]{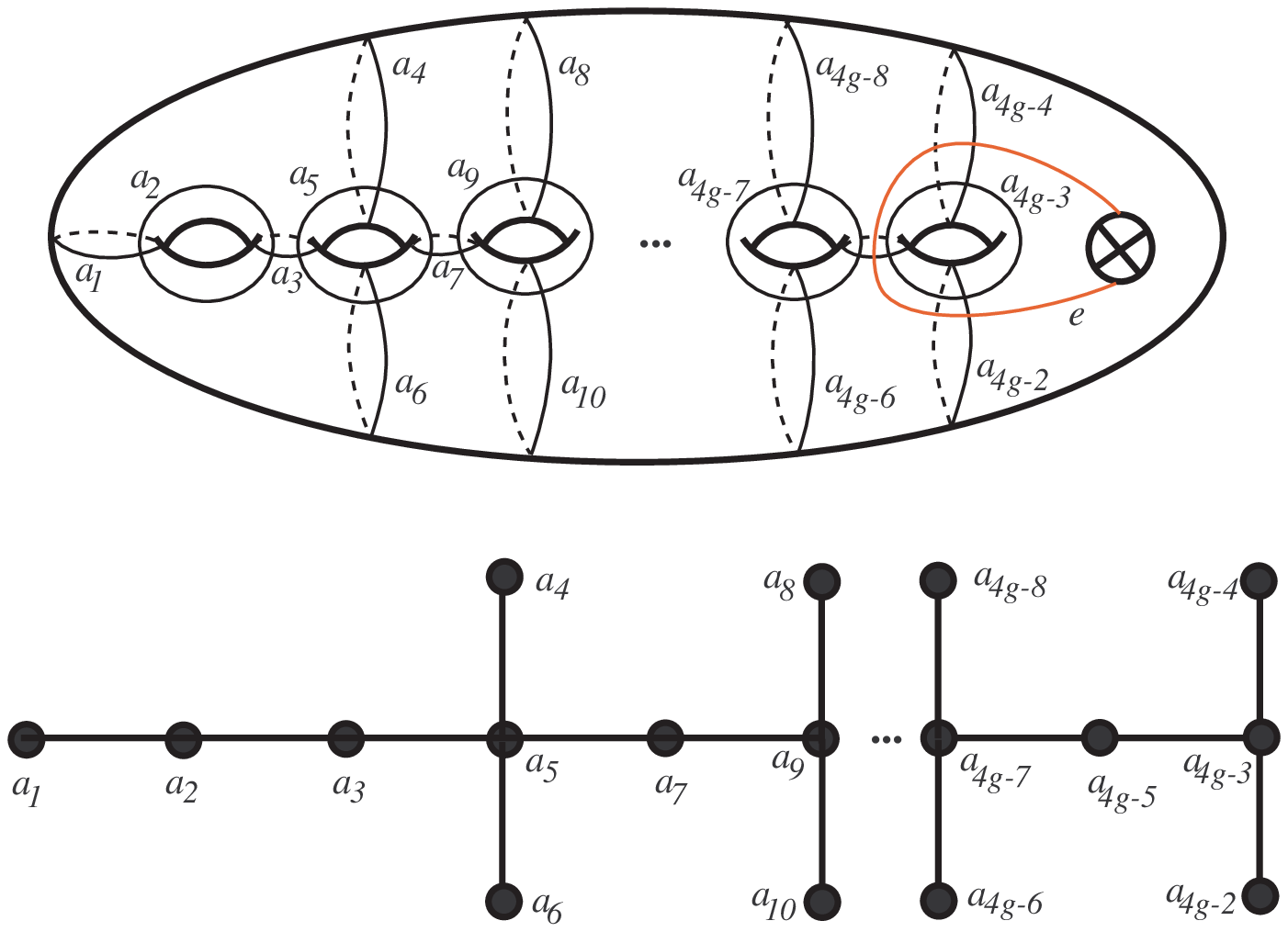}
    \caption{}
    \label{grp2}
  \end{center}
 \end{figure}

It is well known that the Dehn twists $t_{a_i}$, $i=1, \ldots, 4g-2$
and a Y-homeomorphism, say $U$, supported in the tubular neighborhood of
the union of the circles $a_{4g-4}$ and $e$ (which is a copy of one holed
Klein bottle) generate the mapping class group ${\rm {Mod}}(N_{2g+1})$
(c.f. see \cite{K,SM}).

Let $\Phi: {\rm {Mod}}(N_{2g+1}) \to {\rm {Mod}}(N_{2g+1})$ be an
automorphism. By the results of the previous section the image
of this tree, which we denote by $T$, will be again a tree, where
each maximal chain in $T$ bounding a disc will be sent to a maximal
chain in $\Phi(T)$ bounding a disc. Indeed, to get the maximal tree of
circles corresponding to the maximal tree of Dehn twists $\Phi (t_{a_i})$,
$i=1, \ldots, 4g-2$, we can choose a hyperbolic metric on $N_{2g+1}$, and then
let $b_i$ be the unique geodesic in the homotopy class of a circle corresponding
to the Dehn twist $\Phi (t_{a_i})$. In particular, $\Phi(t_{a_{i}}) =
t_{b_{i}}$, for each $i$ (see also \cite{I1}).

The tubular neighborhood of the tree, say $\nu$, of $T$ is an orientable
surface. Choose an orientation on $\nu$ and on each $a_{i}$ so that the
intersections $(a_{i},a_{j})$ ($i \leq j$) is compatible with the
orientation.  Choose an orientation of a tubular neighborhood of the
tree $\{b_{1}, \cdots, b_{4g-2} \}$ and a diffeomorphism  of $a_{1}$
onto $b_{1}$. Keeping track of orientations we can extend this
diffeomorphism to whole of $T$ and hence to a neighborhood of $T$,
call it again $\nu$. This diffeomorphism extends to a diffeomorphism
of the surface since we need to extend through discs and a one holed
real projective plane.

Say, $\psi: N_{2g+1} \to N_{2g+1}$ is one such diffeomorphism. Composing $\Phi$ with
$\psi_{*}^{-1}$, we may assume that $\Phi$ fixes all $t_{a_{i}}$.
Let $U' = \Phi(U)$, where $U$ is as above. Since $U$ commutes with
all $t_{a_{i}}$ for $i \leq 4g-6$, so is $U'$. Since a power of $t_{C}$
is a composition of $t_{a_{i}}$'s, $i \leq
4g-6$, $U'$ commutes with $t_{C}$ also (Figure~\ref{grp2c}). We can isotope $U'$ to fix
$C$ pointwisely. Note that $U'$ cannot interchange the two sides of
a tubular neighborhood of $C$ since the genus of $N_{2g+1}$ is odd. So, we
may assume that $U'$ is identity on a tubular neighborhood of $C$.
So, $U'$ induces a diffeomorphism on $N_{2g+1}^C$ which is the disjoint
union of $\Sigma_{g-1,1}$ and $N_{3,1}$.

\begin{figure}[hbt]
  \begin{center}
    \includegraphics[width=8cm]{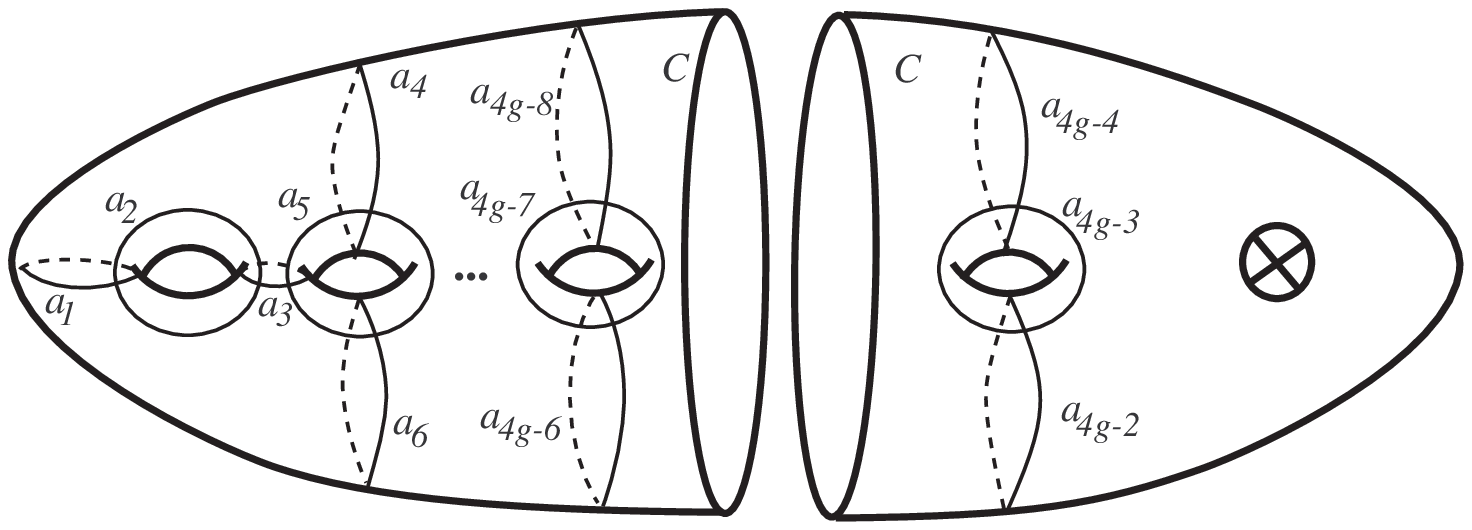}
    \caption{}
    \label{grp2c}
  \end{center}
 \end{figure}

Since $U'$ commutes with all $t_{a_{i}}$'s, $i \leq 4g-6$,
with $t_{C}$, and the center of ${\rm {Mod}}(\Sigma_{g-1,1})$ is $\langle t_{C}
\rangle$ we may assume that after an isotopy, $U'$ is identity on
$\Sigma_{g-1,1}$. Hence, it is supported on $N_{3,1}$.

Now, we need to study ${\rm {Mod}}(N_{3,1})$.  Let us rename the circles
$a_{4g-4}$, $a_{4g-5}$, $a_{4g-3}$ and $a_{4g-2}$, and the
Dehn twists about them by the letters $A_{1}$, $A_{2}$, $A_{3}$ and $B$, respectively.
It is known that ${\rm {Mod}}(N_{3,1})$ has the following presentation ~\cite{SB}:

\begin{figure}[hbt]
  \begin{center}
    \includegraphics[width=5cm]{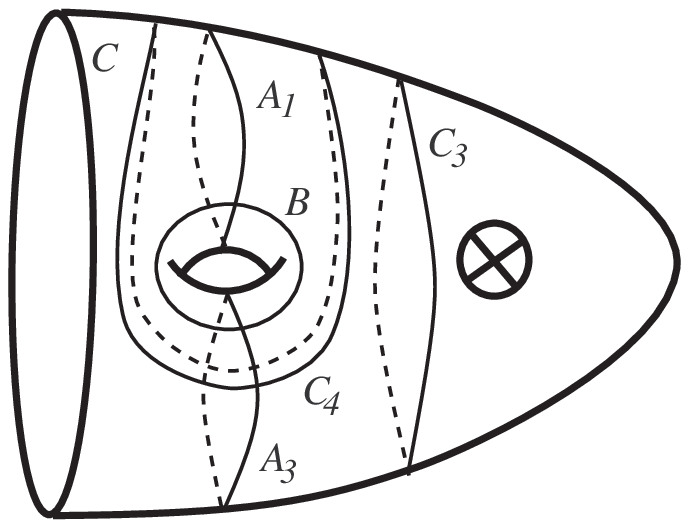}
    \caption{}
    \label{3rp2-1bc}
  \end{center}
 \end{figure}

Generators:$A_{1}$, $A_{3}$, $B$ and $U$.\\
Relations:\\
$1)$ $A_{1} A_{3} = A_{3} A_{1}$\\
$2)$ $A_{1} B A_{1} = B A_{1} B$, and $A_{3} B A_{3} = B A_{3} B$\\
$3)$ $UA_{1}U^{-1} = A_{1}^{-1}$\\
$4)$ $UBU^{-1} = A_{3}^{-1}B^{-1}A_{3}$\\
$5)$ $(A_{3}U)^{2} = (UA_{3})^{2} = (A_{1}^{2}A_{3}B)^{3}$($=C$).

\bigskip
\noindent It turns out that replacing $U$ with $V=A_{3}U$ gives a convenient
set of generators and relations:

Generators: $A_{1}$, $A_{3}$, $B$ and $V$.\\
Relations:\\
$1')$ $A_{1} A_{3} = A_{3} A_{1}$\\
$2')$ $A_{1} B A_{1} = B A_{1} B$, and $A_{3} B A_{3} = B A_{3} B$\\
$3')$ $VA_{1}V^{-1} = A_{1}^{-1}$\\
$4)$ $VBV^{-1} = B^{-1}$\\
$5')$ $V^{2}A_{3} = A_{3}V^{2}$\\
$6')$ $V^{2} = (A_{1}^{2} A_{3} B)^{3}$($=C$).

(Note that $5')$ follows from $6')$).

Since all these generators commute with $t_{C}$, $\Phi$ maps ${\rm {Mod}}(N_{3,1})$
into itself. Indeed, we have $\Phi(A_{1})=A_{1}$, $\Phi(A_{3})=A_{3}$ and $\Phi(B)=B$.

To make things easier we will consider the subgroup of ${\rm {Mod}}(N_{3,1})$
generated by all Dehn twists. We will show that this
subgroup, denoted by $L$, is a normal subgroup of index $2$.

Let  $W=V^{-1}A_{3}V$ which is $VA_{3}V^{-1}=U^{-1}A_{3}U$. Note
that $A_{1}W^{i} = W^{i}A_{1}$ and $W^{i}V = VA_{3}^{i}$, for all $i$. Also,
$B^{-1}WB^{-1} = WB^{-1}W$ and $BW^{-1}B=W^{-1}BW^{-1}$. Set $L =
\langle A_{1}, \,A_{3},\, B,\, W \rangle$. It is easy to see that
this is a normal subgroup of ${\rm {Mod}}(N_{3,1})$ generated by Dehn
twists and hence $V \notin L$. Nevertheless, $V^{2} \in L$ and hence
$L$ has index two in ${\rm {Mod}}(N_{3,1})$. So, $L$ must be the
subgroup generated by all Dehn twists.

Now, any element of ${\rm {Mod}}(N_{3,1})$ which is not a product of a Dehn twist
has the form $DV$ for some $D \in L$. In particular, $\Phi(V)=DV$ for some $D$.
Since $\Phi(V^{2})=V^{2}$, $(V^{2} = (A_{1}^{2} A_{3} B)^{3})$, we
see that $V^{2} = (DV)^{2} = DV DV = DV DV^{-1} V^{2} = DD'V^{2}$,
where $D' = VDV^{-1}$. So, $DD'=id$.

($D'$ is obtained from $D$ by replacing $A_{1}$ with $A_{1}^{-1}$,
$B$ with $B^{-1}$, $W$ with $A_{3}$ and $A_{3}$ with $W$.)

Apply $\Phi$ to $V^{-1}A_{1}V = A_{1}^{-1}$ to get
$V^{-1}D^{-1}A_{1}DV = A_{1}^{-1} = V^{-1}A_{1}V$. It follows that
$A_{1}D = DA_{1}$. Similarly, apply $\Phi$ to $V^{-1}BV=B^{-1}$ to
get $BD = DB$. So, $D$ maps a tubular neighborhood of $A_{1}$ and
$B$ diffeomorphically onto itself. By composing $D$ with
$(A_{1}B)^{3} = (BA_{1})^{3}$ if necessary we may assume that $D$ is
identity on a tubular neighborhood of $A_{1}$ and $B$, where $C_{4}$
is the boundary of the tubular neighborhood of the union of $A_{1}$ and $B$.

The complement of the tubular neighborhood of the union of $A_{1}$ and $B$
is a projective plane with two boundary components; let us call it
$N_{1,2}$. Since ${{\rm {Mod}}}(N_{1,2}) = \mathbb{Z} \bigoplus \mathbb{Z}
= \langle C \rangle \bigoplus \langle C_{4} \rangle$, we get $D =
C^{n}(A_{1}B)^{3m}$ for some $m$ and $n$ $\in \mathbb{Z}$. Now, use
the relation $VDV^{-1} = D^{-1}$ to get $C^{n}(A_{1}B)^{-3m} =
C^{-n}(A_{1}B)^{-3m}$. It follows that $n=0$. So, $D=(A_{1}B)^{3m}$.
Hence, we associate an integer, namely $m$ to each element of
$\rm Out({{\rm {Mod}}}(N_{2g+1}))$. Indeed, this is a group homomorphism
$\rm Out({{\rm {Mod}}}(N_{2g+1})) \to \mathbb{Z}$. To see this first note
the following. Let $\Phi_{1}$ and $\Phi_{2}$ $\in \rm Out({{\rm {Mod}}}(N_{2g+1}))$,
so that $\Phi_{1}$ and $\Phi_{2}$ are identity on
all $t_{a_{i}}$, $A_{i}$, $B$ and $\Phi_{1}(V)= (A_{1}B)^{3m_{1}}V$
and $\Phi_{2} (V)= (A_{1}B)^{3m_{2}}V$. So, $(\Phi_{1} \circ
\Phi_{2})(V)= (A_{1}B)^{3(m_{1}+m_{2})}V$.

To finish the proof of the theorem we need to show that such an
automorphism (if it exist) cannot be an inner automorphism.
To see that this is not an inner automorphism note the following: Any
diffeomorphism which commutes with the all $t_{a_{i}}$ is isotopic to the
identity on the orientable component of $N_{2g+1}-C_3$. So, it
must be isotopic to the identity on the whole surface $N_{2g+1}$.
However, the automorphism induced by the identity diffeomorphism is the identity.
So, it should send $V$ to $V$ also.  It is easy to see in this case that the exponent
$m$ of $A_1B$ must be zero. In particular, the homomorphism
$\rm Out({{\rm {Mod}}}(N_{2g+1})) \to \mathbb{Z}$ is injective.

\end{proof}

\begin{Theorem}\label{MainThm2}
$\rm Out({{\rm {Mod}}}(N_{2g}))$ injects into ${{\rm {Mod}}}(\Sigma_{0,4})$, for $g\geq 3$.
\end{Theorem}

\begin{proof} The idea of this proof is the same as the odd genus case.
However, there are some technical differences.
Consider the maximal tree of Dehn twists $t_{a_i}$, $i=1,\ldots, 4g-5$, in
Figure~\ref{2grp2}, together which a $Y$-homeomorphism supported in the one holed
Klein bottle bounded by the circle $e$ generate the mapping class group ${\rm {Mod}}(N_{2g})$
(c.f. see \cite{K,SM}).

\begin{figure}[hbt]
  \begin{center}
    \includegraphics[width=8cm]{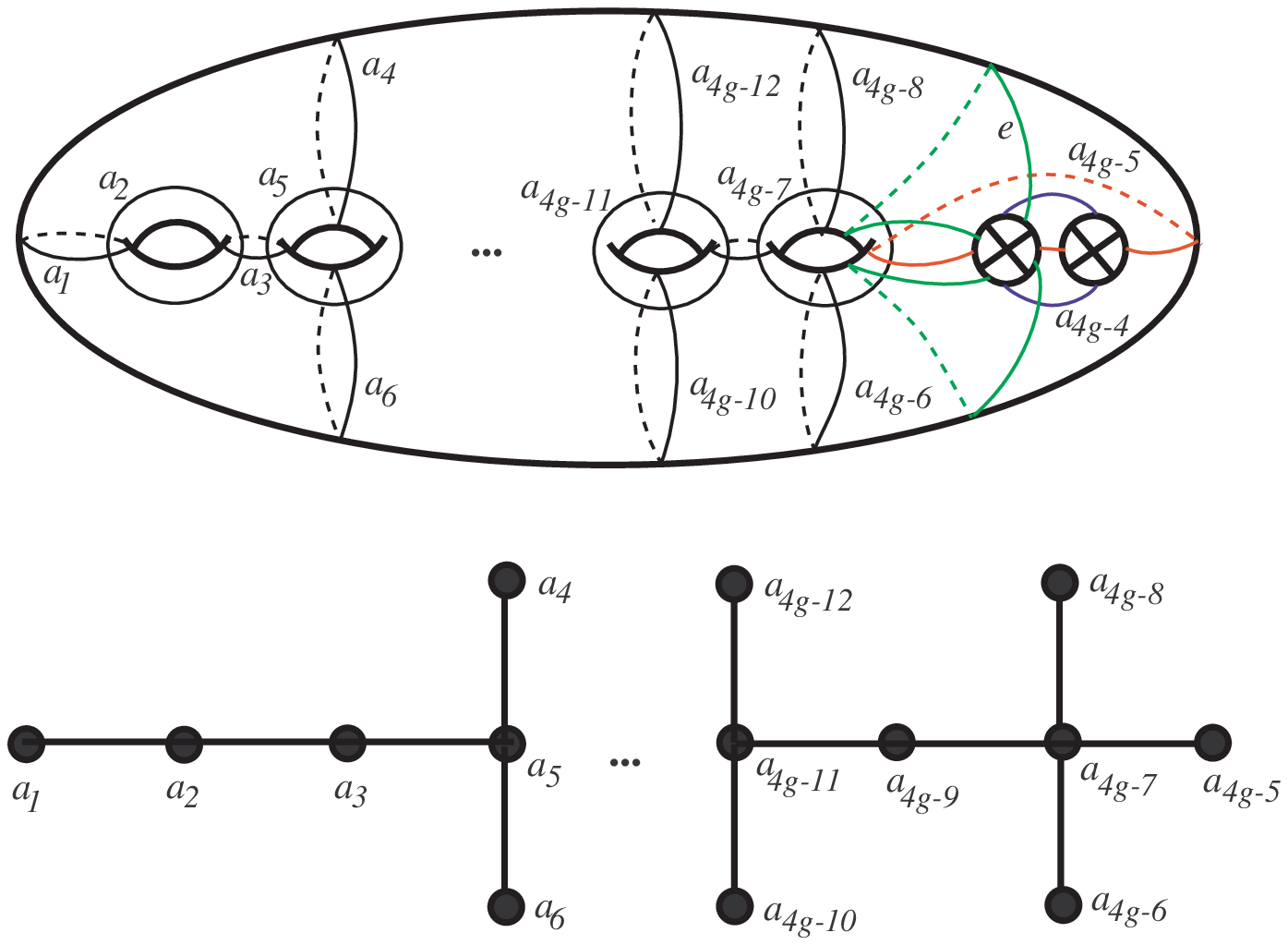}
    \caption{}
    \label{2grp2}
  \end{center}
 \end{figure}

The maximal tree above has tubular neighborhood diffeomorphic to $\Sigma_{g-1,2(g-1)}$.
So, if $\Phi$ is an automorphism of ${{\rm {Mod}}}(N_{2g})$, let
$t_{b_{i}} =\Phi(t_{a_{i}})$, $i=1,\ldots,4g-5$, then $b_{1}, \cdots, b_{4g-5}$ form a
tree as above.  Similar to the odd genus we can find a diffeomorphism
of a tubular neighborhood of the tree formed by $a_{i}$'s ($i=1,\ldots,4g-5$),
call $T$ onto that by $b_{i}$'s, call $T'$. Each maximal subchain bounding a
disc is sent to a maximal subchain bounding a disc. Therefore, we can find a
diffeomorphism $\varphi_{0}: T \cup (2(g-2))D^{2} \rightarrow T'\cup (2(g-2))D^{2}$,
which are both a copies of the orientable surface $\Sigma_{g-1,2}$.
Then $\chi(\Sigma_{g-1,2})=2-2(g-1)-2=2-2g$.
So, $\chi(N_{2g} \setminus \Sigma_{g-1,2})=(2-2g)-(2-2g)=0$. since,
$\Sigma_{g-1,2}$ has two boundary components $N \setminus \Sigma_{g-1,2}$
must be a cylinder and the ends are glued to $\Sigma_{g-1,2}$ so
that the surface becomes nonorientable. In particular, we see that
$N_{2g}$ is obtained by gluing the two boundary components of
$\Sigma_{g-1,2}$. Now, it is easy to extend the diffeomorphism
$\psi_{0}$ to $N$, say $\varphi$. Note that by construction
$\psi_{0}(a_{i})=a_{i}$, for all $i=1,\cdots,4g-5$, where
$\psi_{0}$ is a diffeomorphism of orientable subsurfaces $N_{2g}
\setminus \nu(a_{4g-4})$ onto $N_{2g} \setminus \nu(b_{4g-4})$. So, the
automorphism $\psi_{*}$ from ${{\rm {Mod}}}(N_{2g})$ to ${{\rm {Mod}}}(N_{2g})$
sends $t_{a_{i}}$ to $t_{b_{i}}$, $i=1,\cdots,4g-5$.

Composing $\Phi$ with $\psi_{*}^{-1}$ we may assume that
$\Phi(t_{a_{i}})=t_{a_{i}}$ $i=1,\cdots,4g-5$. Since $t_{a_{4g-4}}$
commutes with all other $t_{a_{i}}$ it should be supported in the
one holed Klein bottle obtained by cutting the surface $N_{2g}$
along $a_{4g-8}$, $a_{4g-7}$ and $a_{4g-6}$. However, there is a
unique nontrivial two-sided circle inside this one holed Klein bottle,
which is $a_{4g-4}$. So, $\Phi(t_{a_{4g-4}})=t_{a_{4g-4}}^{\pm}$.

Since $\Phi$ fixes all $a_{1},\cdots,a_{4g-6}$ any diffeomorphism
supported in $N_{2,2}$ will be sent by $\Phi$ to a diffeomorphism
supported in $N_{2,2}$. So, $\Phi$ induces an automorphism of the
mapping class group of the surface with boundary $N_{2,2}$ (Klein bottle with
two boundary components), obtained by cutting the surface along the circles
$a_{4g-8}$ and $a_{4g-6}$.

Let $N_4$ denote the surface in Figure~\ref{4rp2e}, which is obtained from
$N_{2,2}$ by gluing its boundary components.

Note that any element of
${{\rm {Mod}}}(N_{2,2})$ yields an element of ${{\rm {Mod}}}(N_{4})$. Indeed,
the kernel of the natural homomorphism,
${{\rm {Mod}}}(N_{2,2}) \to {{\rm {Mod}}}(N_{4})$, is $t_{a_{4g-8}}t_{a_{4g-6}}^{-1}$
(c.f. see Lemma 4.1. of \cite{SB}).

To go further we will need a presentation of the mapping class group
${{\rm {Mod}}}(N_{4})$ given in Theorem 2.1. of \cite{SB2}.

\begin{figure}[hbt]
  \begin{center}
    \includegraphics[width=5cm]{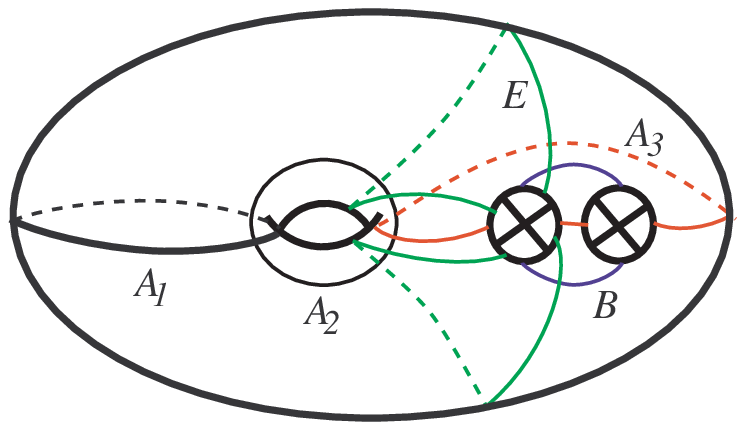}
    \caption{}
    \label{4rp2e}
  \end{center}
 \end{figure}

We will rename some of the circles on $N_{2g}$ so that the presentation in
\cite{SB2} would not be confusing: The circles $a_{4g-8}$, $a_{4g-7}$, $a_{4g-5}$,
$a_{4g-4}$ and the Dehn twists about them will be denoted $A_1$, $A_2$, $A_3$
and $B$, respectively. In \cite{SB2}, it is given that the mapping class group of
$N_4$ is generated by the Dehn twists $A_1$, $A_2$, $A_3$, $B$ and a Y-homeomorphism
called $U_3$.

\begin{Lemma}
The relation $(U_{3}B)^{2} =1$ of ${{\rm {Mod}}}(N_{4})$ gives
$$(U_{3}B)^{2}=(A_{1}')^{k_{1}}(A_{1}'')^{k_{2}}$$ in
${{\rm {Mod}}}(N_{2,2})$, for some $k_{1}$ and $k_{2}$, where $A_{1}'$ and
$A_{1}''$ are the boundary components of $N_4$ cut along $A_1$.
Moreover in the mapping class group of $N_{2g}$, we get
$(U_{3}B)^{2}=(t_{a_{4g-8}})^{k_{1}}(t_{a_{4g-6}})^{k_{2}}$.

\end{Lemma}
\begin{proof}
The relation $(U_{3}B)^{2} =1$ holds in the mapping class group of
$N_4$ (Theorem 2.1 (11) of \cite{SB2}). Now Lemma 4.1 of \cite{SB}
finishes the proof of the first statement. The second statement holds
trivially.
\end{proof}

Let $L$ be normal closure of $t_{a_{i}}$'s and $t_b$. Hence, we have
${{\rm {Mod}}}(N_{2g})/L = \langle U_{3} \rangle \simeq \mathbb{Z}_{2}$.
Since $\Phi$ fixes all $t_{a_{i}}$'s and $t_b$, $\Phi$ induces an automorphism
$\bar{\Phi}$ of ${{\rm {Mod}}}(N_{2g})/ L$ and so that
$\bar{\Phi}(U_{3})= U_{3}$. So, $\Phi(U_{3})=U_{3}X$ for some $X \in L$.
Apply $\Phi$ to $U_{3}A_{1}'=A_{1}'U_{3}$ and
$U_{3}A_{1}''=A_{1}''U_{3}$ to get
$U_{3}XA_{1}'=A_{1}'U_{3}X=U_{3}A_{1}'X$. This implies that
$XA_{1}'=A_{1}'X$ and similarly $XA_{1}'' =A_{1}''X$. Also,
$U_{3}A_{3}= A_{3}^{-1}U_{3}$ yields
$U_{3}XA_{3}=A_{3}^{-1}U_{3}X=U_{3}A_{3}X$. It follows that
$XA_{3}=A_{3}X$. So, $X$ commutes with $A_{1}'$, $A_{1}''$ and
$A_{3}$. Indeed, by similar arguments $X$ commutes with all
$t_{a_{i}}$, $i\leq 4g-5$.

Clearly, $X$ cannot change the two sides of tubular neighborhoods of
$A_{1}'$ and $A_{1}''$ and hence $X$ is identity on a neighborhood
of $A_{1}'$ and $A_{1}''$. The mapping class $X$ cannot change the
two sides of a tubular neighborhood of $A_{3}$ also, since otherwise we would have
$XA_{3}=A_{3}^{-1}X$. So, $X$ is an element of the mapping class group
of the four holed sphere component of the surface $N_{2g}$ cut along the
curves $A_{1}'$,  $A_{1}''$ and  $A_{3}$:

\begin{figure}[hbt]
  \begin{center}
    \includegraphics[width=5cm]{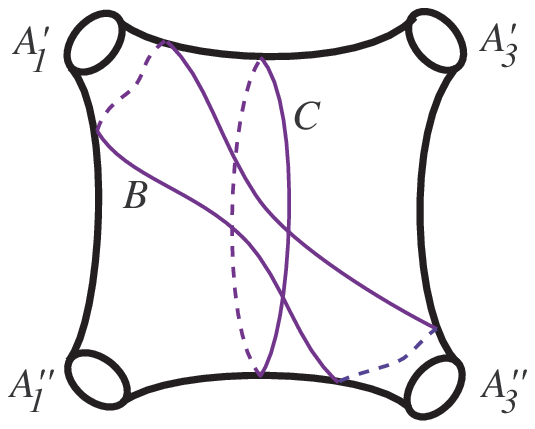}
    \caption{}
    \label{1o}
  \end{center}
 \end{figure}

Hence, $X = (A_{1}')^{n_{1}}(A_{1}'')^{n_{1}'}(A_{3}')^{n_{2}}(A_{3}'')^{n_{2}'}w(B,C)$,
for some integers $n_{1}, \ n_{1}', \ n_{2}, \ n_{2}'$ and word $w(B,C)$, where
$A_{3}'$ and $A_{3}''$ are the new boundary circles obtained by cutting the surface
along $A_3$.

This defines a map
$\rm Out({{\rm {Mod}}}(N_{2g})) \to {{\rm {Mod}}}(\Sigma_{0,4})$, which is indeed a
homomorphism by construction. It is clear that it is also injective.
\end{proof}

\bigskip
\subsection{Some sporadic cases}
In this section we will state some results about outer automorphism groups of some
of the surfaces not covered in the above theorems.

\begin{Proposition}\label{sporadic cases}
{\bf 1)} ${{\rm {Mod}}}(N_2)={\mathbb Z}_2\times {\mathbb Z}_2$ and hence
$\rm Out({{\rm {Mod}}}(N_2))=S_3$, the symmetric group on three letters.\\
{\bf 2)} $\rm Out({{\rm {Mod}}}(N_{2,1}))={\mathbb Z}_2\times {\mathbb Z}_2$.\\
{\bf 3)} $\rm Out({{\rm {Mod}}}(N_3))$ contains a copy of $\mathbb Z_2$.\\
{\bf 4)} $\rm Out({{\rm {Mod}}}(N_{3,1}))$ contains a copy of the
infinite cyclic group $\mathbb Z$.
\end{Proposition}
\begin{proof}
The proof of (1) is immediate. One can prove part (2) directly using the
presentation given in Remark~\ref{Y-hom}.

To prove part (4) recall the presentation of the mapping class group of $N_{3,1}$
and the candidate automorphism defined at the end of Theorem~\ref{MainThm1}:

Generators: $A_{1}$, $A_{3}$, $B$ and $V$.  Relations:\\
$1')$ $A_{1} A_{3} = A_{3} A_{1}$\\
$2')$ $A_{1} B A_{1} = B A_{1} B$, and $A_{3} B A_{3} = B A_{3} B$\\
$3')$ $VA_{1}V^{-1} = A_{1}^{-1}$\\
$4)$ $VBV^{-1} = B^{-1}$\\
$5')$ $V^{2}A_{3} = A_{3}V^{2}$\\
$6')$ $V^{2} = (A_{1}^{2} A_{3} B)^{3}$($=C$), and the automorphism is given by
$\Phi(A_{1})=A_{1}$, $\Phi(A_{3})=A_{3}$, $\Phi(B)=B$ and $\Phi(V)= (A_{1}B)^3V$.
It is easy to check that this map is really an automorphism of the mapping class group.
The fact that any power of it is not an inner automorphism follows from the fact that
any mapping class commuting with the Dehn twists $A_1$, $A_3$ and $B$ should be a power of
the Dehn twist about $C$. However, $C^{-1}VC=V\neq (A_1B)^{3m}V$ for any $m \neq 0$. This
finishes the proof.

Finally, almost the same argument given above for part (4) proves part (3).
\end{proof}

\subsection*{Acknowledgment} This work has been initiated and partially completed during
my stay at the Michigan State University Mathematics Department in the Spring of 2007.
I would like to thank the Mathematics Department and specially to Prof. N. Ivanov and
Prof. S. Akbulut for invitation.

\bigskip
\providecommand{\bysame}{\leavevmode\hboxto3em{\hrulefill}\thinspace}

\end{document}